\documentclass[journal]{IEEEtran}
\usepackage{amssymb, url, amsthm}
\def\R{\mathbb{R}}
\def\E{\mathcal{E}} 
\def\Ea{\mathcal{E}_{\rm ave}} 
\newtheorem{theorem}{\indent {Theorem}}
\begin{document}

\title{On (Non)Supermodularity of Average Control Energy}

\author{Alex~Olshevsky
\thanks{A. Olshevsky is with the Department of Electrical and Computer Engineering and the Division of Systems Engineering, Boston University, 8 St. Mary's St, Boston, MA, 02215, {\tt alexols@bu.edu}}% <-this % stops a space
\thanks{This research was supported by NSF grant ECCS-1351684.}}

\maketitle

\begin{abstract}
 Given a linear system, we consider the expected energy to move from the origin to a uniformly random point on the unit sphere as a function of the set of actuated variables. We show this function is not necessarily supermodular, correcting some claims in the existing literature.
\end{abstract}

% \fi
\IEEEpeerreviewmaketitle

\section{Introduction} This paper is concerned with a  property of the actuator selection problem. Given the linear system 
\[ \dot{x}_i = \sum_{j=1}^n a_{ij} x_j, ~~~ i = 1, \ldots, n,\] the simplest actuator selection problem asks for the smallest possible set of variables to affect with an input in order to achieve a prespecified  control objective. Typical control objectives include controllability of the resulting system or the ability to steer the system subject to an energy constraint. 

Formally, if we choose to affect the set of  variables $\{x_i ~|~ i \in I \}$ then the resulting system-with-input is
\begin{eqnarray} \dot{x}_i & = & \sum_{j=1}^n a_{ij} x_j + u_i, ~~~~~ i \in I. \nonumber \\ 
\dot{x}_i & = & \sum_{j=1}^n a_{ij} x_j, \mbox{      } i \notin I. \label{varact}
\end{eqnarray} and the goal is to choose the set $I$ as small as possible while still satisfying some control objective. More complex versions of actuator selection problem might not allow one to directly affect each variable; rather, one instead assumes that the system can only be affected in several distinct ``sites'' and affecting each site affects some subset of the variables all at once. 

The actuator selection problem received some attention recently (e.g., \cite{mincont, lygeros, ali}), due to the emergence of recent interest in large-scale systems, for example in power networks or systems biology. It may be impractical or uneconomical to steer  large systems by affecting every, or even most, of the variables, and consequently it is natural to ask if the system can be efficiently steered by affecting only very few select variables.

A key property for actual selection problems is { supermodularity}. A formal definition can be found in the next section, but, roughly speaking this is the property that affecting variables runs into diminishing returns; that is to say, affecting a certain variable has less impact on the control objective if more variables have already been affected. 

Supermodularity is important because it can lead to algorithms with rigorous approximation guarantees. For example, an approximate algorithm for actuator selection to render the system controllable based on supermodularity of the dimension of the controllable subspace was given in \cite{mincont}\footnote{Note that although \cite{mincont} did not use the words "supermodularity" or "submodularity,"  some of the steps of the proofs were formulations of  this property.}. 

Supermodularity of a number of a control objectives was studied in the recent papers \cite{lygeros} and \cite{ali}. Specifically, one of the control objectives studied in \cite{lygeros} was the trace of the inverse of the controllability Gramian, which has the interpretation of being proportional to the expected energy to move from the origin to a  random point on the unit sphere (we will refer to this as the {\em average control energy}). It was claimed in \cite{lygeros} that, for a stable system, average control energy is a supermodular function of the set of affected sites. Using similar arguments, the later paper \cite{ali} claimed that (an arbitrarily small perturbation of) average control energy is a supermodular function of the set of affected variables.

{ The purpose of this note is to show that average control energy is not always supermodular, contrary to what is claimed in \cite{lygeros} and \cite{ali}. In other words, there exists a (stable, symmetric) linear system and two sets of variables, $I_1 \subset I_2$ such that average control energy decreases {\em more} when a certain variable is added to the bigger set of actuated variables $I_2$, as compared to the scenario when the same variable is added to the smaller set $I_1$. 

The remainder of this paper is organized as follows. In Section \ref{sec:def}, we give the basic definitions used in the remainder of the paper. The subsequent Section \ref{sec:counter} contains the constructions of linear systems for which average control energy is not supermodula. Finally, Section \ref{sec:concl} concludes with some brief remarks.

\subsection{Notation}  We use the standard notation of letting $e_i$ denote the $i$'th basis vector and $I_k$ to denote the $k \times k$ identity matrix. For a matrix $M$, we will use $M'$ to denote its transpose. The complement of a set $S$ will be denoted by $S^c$. The notation $1_k$ will be used for the column vector of all ones in $\R^k$. Finally, a matrix is called strictly stable if all of its eigenvalues have negative real parts. 

\section{Basic definitions\label{sec:def}} 

\subsection{Average control energy of linear systems} 
Given the linear system 
\begin{equation} \label{linear} \dot{x} = Ax + Bu, \end{equation} and an initial state $x_0$ along with a final state $x_{\rm f}$, we define the control energy $\E( A, B, x_0 \rightarrow x_{\rm f}, T) $ to be the minimal energy $\int_0^T ||u(t)||_2^2 ~ dt$ among all inputs $u:[0,T] \rightarrow \R$ which result in $x(T)=x_{\rm f}$ starting from $x(0) = x_0$. If there is no input which results in $x(T)=x_f$ when $x(0)=x_0$, we will adopt the convention that $\E(A, B, x_0 \rightarrow x_f)$ is infinite.

The quantity $\E(A, B, x_0 \rightarrow x_{\rm f})$ measures the difficulty of steering the system from $x_0$ to $x_f$; obviously it will depend on both the starting point $x_0$ and the final point $x_f$. One way to obtain a measure of the ``difficulty of controllability'' of the entire system is to consider the energy involved in moving the system from the origin to a uniformly random point on the unit sphere, namely \[ \Ea(A,B, T) :=  \int_{||y||_2 = 1} ~\E(A,B, 0 \rightarrow y, T) ~dy. \]

It is easy to see that this quantity can be written in terms of controllability Gramian. Indeed, first we define the controllability Gramian $W(T)$ in the usual way as  \begin{equation} \nonumber W(A,B,T) := \int_0^{T} e^{At} B B' e^{A' t} ~dt, \end{equation} where we will allow $T$ to be equal to $+\infty$ with the proviso that $W(+\infty)$ is well-defined only as long as the matrix $A$ is strictly stable. It is then not difficult to see that
\[ \Ea(A,B,T) = \frac{1}{n} {\rm tr} \left[ W(A, B, T)^{-1} \right]. \] Moreover, if $W(A,B,T)$ is not invertible then $\Ea(A,B,T)$ is infinite.

\subsection{The actuator selection problem} Before giving a formal statement of the actuator selection problem, let us introduce some notation. First, we will need notation for the dimensions of $A$ and $B$; specifically, let us suppose $A \in \R^{n \times n}$ and $B \in \R^{n \times m}$. Then given a set $S \subset \{1, \ldots, m\}$, we denote $B(S)$ to be the matrix in $\R^{n \times |S|}$ composed of the columns of $B$ corresponding to indices in $S$. For example, if $B=I_3$ (the $3 \times 3$ identity matrix), then \[ B(\{1,2\})  = \left( \begin{array}{cc} 1 & 0 \\ 0 & 1 \\ 0 & 0 \end{array} \right). \] We then define 
\[ \Ea(A,B, T,S) := \Ea(A,B(S),T). \] In other words, $\Ea(A, B, T, S)$ is the average control energy at time $T$ when using only the columns of $B$ corresponding to the indices in the set $S$. 

There are many versions of actuator selection problems, but the ones we consider here ask to optimize the function $\Ea(A,B,T,S)$ for fixed $A,B,T$ as a function of the set $S$. For example, a natural problem is to ask for $S$ of smallest  cardinality meeting the energy constraint $\Ea \leq c$ for some real number $c$. 

In the simplest and most natural case, $B$ is the $n \times n$ identity matrix; in that case, we may think of choosing $S$ as corresponding to actuating the variables of the differential equation $\dot{x}=Ax$ as in Eq. (\ref{varact}).
More generally, affecting a system in a given ``site'' may affect a group of variables all-at-once; this is appropriately modeled by a  general matrix $B$, where the choice of adding an index $i$ to $S$ involves affecting all the variables in the $i$'th column of $B$.

\subsection{Set functions}  A function $f: 2^{\{1, \ldots, m\}} \rightarrow \R$ is called nonincreasing if $S_1 \subset S_2$ implies $f(S_1) \geq f(S_2)$. A set function is  called supermodular if $S_1 \subset S_2$ and $a \notin S_2$ implies that 
\begin{equation} \label{smod} f(S_1) - f(S_1 \cup \{a\}) \geq f(S_2) - f(S_2 \cup \{a\}). \end{equation} Intuitively, if the function $f$ is supermodular then adding element $a$ decreases the function less if it is added to the bigger set $S_2$ as compared to the smaller set $S_1$.

A set function is called submodular if its negation is supermodular. 

%Similarly, one can define the notion of a nondecreasing %set function; furthermore, a nondecreasing set function %will be called submodular if its negation is %supermodular.

\section{Average Control Energy May Not Be  Supermodular\label{sec:counter}}

Throughout this section, we will investigate the setup where $A,B,T$ are fixed and $\Ea(A,B,T,S)$ is considered as a function only of the set $S$. It is quite easy to see this function is nonincreasing, i.e., average control energy cannot increase when we actuate more places. 

As discussed earlier, one might further guess that $\Ea(A,B,T,S)$ would be a supermodular function of $S$. Indeed, it seems quite intuitive that the gain from actuating any specific variable runs into diminishing returns as other variables become actuated. Strangely enough, it turns out that this intuition is not correct and we now turn to the main point of this note, which is to construct counterexamples for this intuition. 

\subsection{(Non)supermodularity  of average control energy for strictly stable matrices}

We begin with an example showing that average control energy may not be supermodular even if the system is strictly stable. 

\begin{theorem} There exists a $2 \times 2$ matrix $A$ and a $2 \times 5$ matrix $B$ such that \begin{enumerate} \item $A$ is strictly stable. 
\item $\Ea(A,B(S),+\infty)$ is finite for all nonempty $S$. 
\item  $\Ea(A,B(S), +\infty)$ is not a supermodular function of $S$. \end{enumerate} \label{first} \end{theorem}

This theorem contradicts Theorem 5 of \cite{lygeros}, which claims that $-\Ea(A,B(S),+\infty)$ is a submodular function of $S$ under the assumptions that (i)  $\Ea(A,B(S),+\infty)$ is finite for all $S$ (ii) $A$ is stable. We mention that later in this paper we will use Theorem \ref{first} to construct a counterexample where $A$ is $6 \times 6$ and $B$ is the $6 \times 6$ identity matrix.

The proof of Theorem \ref{first}, given next, relies primarily on  calculation; since the controllability Gramians involved are $2 \times 2$, this can be done explicitly (though somewhat laboriously). After the proof is concluded, we will discuss the motivation behind the specific choices made within the course of the proof. 

\begin{proof}[Proof of Theorem \ref{first}] 

We first observe that if we can find matrices $A$ and $B$ satisfying the assumptions of the theorem and sets $S_1, S_2, \Delta$ with $S_1 \subset S_2$ and $\Delta \subset S_2^c$ such that 
$\Ea(A,B(S_1),+\infty) - \Ea(A,B(S_1 \cup \Delta))$ is less than $\Ea(A,B(S_2),+\infty) - \Ea(A,B(S_2 \cup \Delta), +\infty)$  then we will have shown that $\Ea(A,B(S),+\infty)$ is not a supermodular function of $S$. Indeed, this is almost identical to the definition of supermodularity with the inequality reversed, with the exception that the set $\Delta$ can now have more than a single element. However, if $\Ea(A,B(S),+\infty)$ were supermodular, we could add the elements of $\Delta$ one by one to $S_1$ and $S_2$ respectively, and obtain that the right-hand side  is at most the left-hand side in the above inequality. 

We next describe how to choose $A,B,S_1, S_2, \Delta$ such that the above inequality holds. We mention again that choices will appear somewhat arbitrary; however, after the proof is over we will explain the intuition behind them.

The matrix $B$ will be $2\times5$ and the matrix $A$ will be $2 \times 2$. Furthermore, let us adopt the notation $b_1, \ldots, b_5$ for the five columns of $B$; each $b_i$ belongs to $\R^2$. 

First, we will set $A=(-1/2)I_2$. Observe that as a consequence of 
this, $$\Ea(A,B,+\infty,S) = \sum_{i \in S} b_i b_i'.$$

Now the columns of $B$ will be determined as follows. Letting $$W_{\rm init} = \left( \begin{array}{cc} 2^8 & 0 \\ 0 & 3 \cdot 2^9 
\end{array} \right)$$ and define $b_1,b_2$ to be the vectors with the property that \[ b_1 b_1' + b_2 b_2' = W_{\rm init}, \] specifically $b_1 = 2^4 e_1, b_2 = \sqrt{3 \cdot 2^9} e_2$. Similarly, let $$W_{\Delta} =  \left( \begin{array}{cc} 5 \cdot 2^9 & -3 \cdot 2^9 \\ -3 \cdot 2^9 & 2^{10} \end{array} \right)$$ and 
let $b_3, b_4$ be vectors such that 
\[ b_3 b_3' + b_4 b_4' = W_{\rm \Delta}. \] Such vectors exist because $W_{\Delta}$ is positive definite (this can be verified by looking at its two principal minors). Finally, we set $b_5 = [1 ~~2^6]'$.

We now claim that  
\begin{footnotesize}
\begin{equation} \label{eineq} \Ea(\{1,2\}) - \Ea(\{1,2,3,4\}) < \Ea( \{1,2,5\}) - \Ea(\{1,2,3,4,5\}), \end{equation} \end{footnotesize}
 where $\Ea(S)$ is used as shorthand for $\Ea(A,B,+\infty,S)$ for the choices of $A,B$ described above.

Indeed, since all the matrices are $2 \times 2$, we can compute both sides exactly. Using the identity \begin{equation} \label{22trace} {\rm tr} \left( \begin{array}{cc} a & b \\ c & d \end{array} \right)^{-1} = \frac{a+d}{ad-bc}, \end{equation} we compute expressions for the left- and right- hand sides of Eq. (\ref{eineq}) in Eq. (\ref{lefth}) and Eq. (\ref{righth}). A bit of calculation reveals that number on the right-hand side of Eq. (\ref{lefth}) equals  ${49}/{14,208}$ while the number of the right-hand side of Eq. (\ref{righth}) equals ${82,017,217}/{23,373,975,296}$, and that the second of these numbers is bigger than the first.   Thus Eq. (\ref{eineq}) follows. \begin{figure*}[!t] \begin{equation} \label{lefth} {\rm tr} \left[ \left( \begin{array}{cc} 2^8 & 0 \\ 0 & 3 \cdot 2^9 \end{array} \right)^{-1} - \left(\left( \begin{array}{cc} 2^8 & 0 \\ 0 & 3 \cdot 2^9 
\end{array} \right) + \left( \begin{array}{cc} 5 \cdot 2^9 & -3 \cdot 2^9 \\ -3 \cdot 2^9 & 2^{10} \end{array} \right)  \right)^{-1} \right] = \frac{7}{2^9 \cdot 3} - \frac{3 \cdot 7}{2^9 \cdot 37}  \end{equation}  \end{figure*}  \begin{footnotesize} \begin{figure*}[!t] \begin{scriptsize}
\begin{equation} \label{righth} {\rm tr} \left[ \left( \left( \begin{array}{cc} 2^8 & 0 \\ 0 & 3 \cdot 2^9 
\end{array} \right) + \left( \begin{array}{cc} 2^0 & 2^6 \\ 2^6 & 2^{12} \end{array} \right) \right)^{-1} - \left(\left( \begin{array}{cc} 2^8 & 0 \\ 0 & 3 \cdot 2^9 
\end{array} \right) + \left( \begin{array}{cc} 2^0 & 2^6 \\ 2^6 & 2^{12} \end{array} \right) + \left( \begin{array}{cc} 5 \cdot 2^9 & -3 \cdot 2^9 \\ -3 \cdot 2^9 & 2^{10} \end{array} \right)  \right)^{-1} \right] = \frac{3 \cdot 13 \cdot 151}{2^9 \cdot 2819} - \frac{9473}{2^9 \cdot 7^2 \cdot 661}  \end{equation} \end{scriptsize} \end{figure*} \end{footnotesize}

We next verify the conditions of the theorem. The matrix $A$ is clearly strictly stable; unfortunately, it is not true that $W(A,B(S),+\infty)$ is always invertible.

 To fix this define  $$ A_{\epsilon} = -\frac{1}{2} I_2 + \epsilon 1_2 1_2',$$ where, recall, $1_2$ is the vector of all-ones in $\R^2$. When $\epsilon$ is positive but small enough, we have that $A_{\epsilon}$ is strictly stable;  furthermore we argue that when $\epsilon$ is small enough,  $W(A,B(S)), +\infty)$ is then invertible for all nonempty $S$. Indeed, by the standard arguments it suffices to show that the controllability matrix $[B(S) ~~ A_{\epsilon} B(S)]$ is invertible for all sets $S$ which contain only a single element.
 Now since $A_{\epsilon}$ is $2 \times 2$, the only way the matrix $[b ~~A_{\epsilon} b]$ could fail to be invertible is if $b=0$ or $b$ was an eigenvector of $A_{\epsilon}$.  Observe that the eigenvectors of $A_{\epsilon}$ are always $[1, ~1]', [1, ~-1]'$, both of which we argue were avoided in our choice of the columns of $B$. Indeed, clearly the first, second, and fifth columns of $B$ are clearly not proportional to either of $[1, ~1]', [1, ~-1]'$. As for the third and fourth columns, these were defined through the property that $b_3 b_3' + b_4 b_4' = W_{\Delta}$, so they can be chosen to be proportional to the eigenvectors of $W_{\Delta}$, and it is easy to verify that neither $[1, ~1]'$ nor $[1, ~-1]'$ is an eigenvector of $W_{\Delta}$.

Finally, since $W(A,B(S),+\infty)$ is a continuous function of the entries of $A$ over the set of strictly stable matrices\footnote{This follows because for strictly stable $A$, $W(A,B(S),+\infty)$ is the unique solution of the linear system equations $A W + W A' + B B'=0.$}, we have that a counterexample may be picked by choosing $\epsilon$ small enough. 

\end{proof}

\noindent {\bf Remark:}  A matrix $B$ constructed according to the above proof is 
\[ B \approx \left( \begin{array}{ccccc} 16 & 0 & 50.5964 &  0 &  1 \\
0 & 39.1918 & -30.3579 & 10.1193 & 64
\end{array} \right) \] Entering this matrix into MATLAB with $A = -(1/2) I_2$ and computing the controllability Gramians
using the ``gram'' command gives that 
\begin{eqnarray*} \Ea(\{1,2\}) - \Ea( \{1,2,3,4\}) & \approx & 0.003449 \\
 \Ea(\{1,2,5\}) - \Ea(\{1,2,3,4,5\}) & \approx & 0.003509
 \end{eqnarray*} providing a numerical verification of non-supermodularity on an example.
 
 \bigskip
 
 \noindent {\bf Remark:} Although we could have simply noted the MATLAB results of the previous remark, the purpose of this paper is to construct counterexamples with rigorous proofs. The proof of Theorem \ref{first} can be verified by a human being (albeit one ready to do several lengthy multiplications and divisions) and does not require any assumption on the correctness of MATLAB's source code, nor is it vulnerable to concern about the effect of round-off error in MATLAB's calculations.

\subsection{Motivation for the proof of Theorem \ref{first}.} We now explain how the counterexample of Theorem \ref{first} was constructed. The final  part -- namely, the perturbation by adding $\epsilon 1_2 1_2'$ to satisfy the conditions of the theorem -- is intuitive enough, as is the choice of $A=(-1/2)I$. The only unintuitive part is the choice of the matrix $B$, in particular through the matrices $W_{\rm init}$ and $W_{\rm Delta}$.  To motivate these choices, we begin by tracing out the problem in the arguments  the papers \cite{lygeros} and \cite{ali}. 

\bigskip

\bigskip

 For simplicity, let us adopt the notation $W(S)$ for $W(A,B(S),+\infty)$. The starting observation in \cite{lygeros} is that 
\[ W(S) = \sum_{i \in S} W(\{i\}), \] which is a  consequence of the definition of the controllability Grammian. Thus given $S_1 \subset S_2$ and $a \notin S_2$, supermodularity of average control energy is equivalent to the following inequality
\begin{eqnarray}  {\rm tr} \left[ W(S_1)^{-1} - \left( W(S_1) + W(\{a\}) \right)^{-1} \right] & \geq & \nonumber \\  {\rm tr} \left[ W(S_2)^{-1} - \left( W(S_2) + W(\{a\}) \right)^{-1} \right] && \label{incorrect} \end{eqnarray}
 The authors of \cite{lygeros} adopt the following approach. They define 
\[ W(\gamma) = W(S_1) + \gamma \left( W(S_2) - W(S_1) \right) \] and consider the function
\[ f(\gamma) = {\rm tr} \left[ W(\gamma))^{-1} - \left( W(\gamma) + W((\{a\}) \right)^{-1} \right]. \] If it could be shown that $f(\gamma)$ is nonincreasing over the range $[0,1]$, this would imply Eq. (\ref{incorrect}) and complete the proof. 
To show this, one can compute the derivative $f'(\gamma)$ which, via a standard computation, turns out to be equal to
\begin{footnotesize}
\begin{equation} \label{gammaderivative}f'(\gamma) = {\rm tr} \left[ \left(   \left( W(\gamma) + W(\{a\}) \right)^{-2} - W(\gamma)^{-2} \right) (W(S_2) - W(S_1)) \right] \end{equation}
\end{footnotesize} So far the argument is correct, and the error in \cite{lygeros} comes in the subsequent assertion that indeed  $f'(\gamma) \leq 0$. For this assertion to hold, we would need to have  that \begin{equation} \label{squares} \left( W(\gamma) + W(\{a\}) \right)^{-2} - W(\gamma)^{-2} \preceq 0,\end{equation} and use the fact that the product of nonnegative definite and a nonpositive definite matrix has nonpositive trace. Unfortunately, Eq. (\ref{squares}) is not always correct. 

Indeed, to obtain Eq. (\ref{squares}), one must rely on the statement that ``$U \preceq V$ implies $U^2 \preceq V^2$.'' Though it is somewhat counter-intuitive, in fact this implication may not hold. This is a rather subtle point, as $U \preceq V$ does imply $U^{\alpha} \leq V^{\alpha}$ when $\alpha \in [0,1]$ (this is the so-called L$\ddot{o}$wner-Heinz inequality) but in general this implication does not necessarily hold if $\alpha > 1$ (see \cite{carlen} for more on this, and we will also discuss it below). The proof of a related assertion in \cite{ali} is  similar and suffers from the same problem.

\bigskip

\bigskip

The above discussion presents a natural way to construct a counterexample to supermodularity of average control energy: we will work backwards from the proof above; first we will construct a counterexample to the last inequality, then the one before it, and so on until we reach a counterexample to the supermodularity. 

Specifically, first we will start with matrices $U$ and $V$ such that $U \preceq V$ but it is not true that $U^2 \preceq V^2$. Secondly, we will use these matrices to come up with positive definite $2 \times 2$ matrices $W_1, W_2, W_3$ with $W_1 \preceq W_2$ such that the function \begin{footnotesize}
\begin{equation} \label{gdef} g(\gamma) = {\rm tr} \left[  \left( W_1 + \gamma (W_2 - W_1)  \right)^{-1} - \left( W_1 + \gamma (W_2 - W_1) + W_3 \right)^{-1} \right] \end{equation} \end{footnotesize} satisfies $g'(0)> 0$. Comparing this to Eq. (\ref{gammaderivative}) we see that after we write the matrices $W_1, W_2, W_3$ are controllability Grammians, we will be in the situation where moving in the direction of the bigger set leads to a larger decrease when adding the variables giving rise to $W_3$ as the controllability Grammian. In the last step, we will use this to construct a counterexample. 

We next discuss the details of each of these steps. 

\begin{itemize} \item There are a number of choices of $U,V$ such that $U \preceq V$ but $U^2 \not \preceq V^2$ that can be taken from the literature (see e.g., \cite{carlen}); here, we belabor the point a little by describing how to choose $U,V$ corresponding to a certain geometrical intuition. 

It is natural to associate positive semi-definite matrices with ellipses; to each positive definite matrix $M$ we associate the ellipse $E(M) = \{ x ~|~ x^T M^{-1} x \leq 1 \}$. This ellipse has principal axes corresponding to the eigenvectors and the length of each principal axis is the square root of the corresponding eigenvalue. The inequality $M_1 \preceq M_2$ is then the same as the inclusion $E(M_1) \subset E(M_2)$. Furthermore, $E(M_1^2)$ can be obtained by stretching each principal axis of $E(M)$ proportionally to the square root of the corresponding eigenvalue. 

We thus need to come with two ellipses $E(V) \subset E(U)$ such that after stretching the principal axes as above, $E(V^2) \not \subset E(U^2)$. One solution is to take $E(V)$ very narrow along the x-axis and wide along the y-axis, while taking $E(M)$ tilted at a 45 degree angle. This leads, after some numerical experimentation, to 
\[ U = \left( \begin{array}{cc} 10 & 6 \\ 6 & 10 \end{array} \right),  ~~~ V = \left( \begin{array}{cc} 80 & 0 \\ 0 & 11 \end{array} \right). \]

\bigskip

\item 

Our goal in this step is to come up with positive-semidefinite matrices $W_1, W_2, W_3$ with $W_1 \preceq W_2$ such that  
\begin{equation} g'(0) = {\rm tr} \left[  \left( \left( W_1 + W_3 \right)^{-2} - W_1^{-2}  \right) (W_2 - W_1)   \right] > 0 \label{gineq} \end{equation}

Our starting point is the matrices $U$ and $V$ we obtained in the previous step. First, we come up with a positive semi-definite matrix $Z$ such that  
\begin{equation} \label{uvineq} {\rm tr}((U^2 - V^2)Z) > 0. \end{equation} This is easily done by setting $Z = z z'$ where $z$ is an eigenvector of $U^2-V^2$ corresponding to the positive eigenvalue. 

Next,  
inspecting Eq. (\ref{uvineq}) and Eq. (\ref{gineq}), it is immediate that we can set 
\begin{equation} \label{wdet} W_1 = V^{-1}, W_2  = V^{-1} + Z, W_{3} = U^{-1} - V^{-1}. \end{equation} We are done, as this choice reduces Eq. (\ref{gineq}) to Eq. (\ref{uvineq}). 

\bigskip

\item Now that we have $W_1, W_2, W_3$ such that $g'(0)>0$ (where the function $g(\cdot)$ was defined in Eq. (\ref{gdef})), two issues need to be addressed to construct a counterexample.

First, we need to be able to write these matrices as controllability Grammians. To do this, we first choose  $A=(-1/2)I$. As we have observed earlier in the proof of Proposition \ref{first}, this choice of $A$ with any $B$ leads to $W(S) = \sum_{i \in S} b_i b_i^T$, where $b_i$ is the $i$'th column of $B$. Now consider the matrix  $W_1$ that we have chosen:  we simply obtain $B_1$ such that $B_1 B_1^T = W_1$ via the Cholesky decomposition and put the columns of $B_1$ as the first and second column of the matrix $B$.  Similarly, letting $B_{\rm diff} B_{\rm diff}^T = W_2 - W_1$, we can put the columns of $B_{\rm diff}$ as the next columns of $B$. Finally, $W_3$ will likewise determine the last columns of $B$.
 
Second, with this choice of the matrix $B$, a successful counterexample is equivalent to the assertion $g(1) > g(0)$; this is, of course, not implied by the assertion that $g'(0)>0$. However, $g'(0)>0$ does imply that $g(\widehat{\gamma}) > 0$ for some small enough $\widehat{\gamma}$. That turns out to be enough for us, as we can simply replace $W_2$ by $W_1  + \widehat{\gamma} (W_2 - W_1)$.

\end{itemize}

The example of Theorem \ref{first} was constructed by following these steps, experimenting with the value of $\widehat{\gamma}$, and finally scaling/rounding the resulting numbers. The scaling/rounding is why the controllability Grammians in Theorem \ref{first} have integer entries.

\subsection{(Non)supermodularity for direct variable actuation} 
We now turn to the special case when $B$ is the identity matrix. As we have previously remarked, this case has a special significance as it corresponds to choosing which variables can be directly actuated with an input. 

Before stating our result, we introduce the following convention. Suppose that $f$ is a function from $2^{\{1, \ldots, n\}}$ to $\R \cup \{+\infty\}$. We will say that $f$ is supermodular if Eq. (\ref{smod}) holds for all choices of $S_1 \subset S_2,a \in S_2^c$ such that every term in Eq. (\ref{smod}) is finite.

We now have the following theorem.

\smallskip

\begin{theorem} There exists a strictly stable, symmetric matrix $A \in R^{6 \times 6}$ such that $W(A,I_6(S),+\infty)$ is not a supermodular function of $S$. \label{secondthm}
\end{theorem}

Recall here our notation: $I_6$ refers to the $6 \times 6$ identity matrix and $I_6(S)$ is the matrix in $\R^{6 \times |S|}$ obtained by picking the columns corresponding to the set $S \subset \{1, \ldots, 6\}$.  

 Theorem \ref{secondthm} contradicts Proposition 2 in \cite{ali}. Indeed, Proposition 2 in \cite{ali} asserts that the function ${\rm tr}\left[ \left( W(A,I(S),t) + \epsilon I \right)^{-1} \right]$ is supermodular, for $\epsilon$ small enough and any $t$. Taking the limit first as $t \rightarrow \infty$ and then as $\epsilon \rightarrow 0$, we obtain a contradiction with Theorem \ref{secondthm}.

We next prove Theorem \ref{secondthm} by showing how the counterexample of Theorem \ref{first} can be embedded into six dimensions.

\begin{proof}[Proof of Theorem \ref{secondthm}]  

Our first observation is that the change of variables $y = P x$ does not change the control energy as long as $P$ is orthonormal. Consequently, it suffices to construct a linear system with an orthonormal input matrix such that $W(\cdot, \cdot,,+\infty)$ is not supermodular, and then Theorem \ref{secondthm} will follow via a change of variables. 

Take the matrix $B$ constructed in that proposition. It is a $2 \times 5$ matrix; add one element to each row such that: (i) the two rows are orthogonal (ii) the two rows have identical norm\footnote{This is always possible, since, if the two elements to be added (one to each row) are denoted as $\alpha$ and $\beta$, then they must satisfy $\alpha \beta = c_1, \alpha^2 - \beta^2 = c_2$, where $c_1$ is the negative inner product of the first two rows of $B$, and $c_2$ is the difference in the squared norm of the first two rows. Since the function $\alpha^2 - (c_1/\alpha)^2$ contains all of $\R$ in its range if $c_1 \neq 0$, such $\alpha$ and $\beta$ can always be found. Finally, it is immediate to verify that indeed $c_1 \neq 0$ (i.e., the first two rows of the matrix $B$ from Theorem \ref{first} are not orthogonal).}. After this is done, normalize both rows to have unit norm. 
We now have a $2 \times 6$ counterexample whose rows are orthonormal. Call the resulting matrix $B_1$. 

Define $A_1(K) = {\rm diag}(-K/2, -K/2, -4, -3, -2, -1)$. Let $B_2$ be the $6 \times 6$ matrix whose first two rows equal $B_1$ and the rest of the rows are equal to zero.  Finally we create $B_3$ by filling in random standard normal entries for the last four rows of $B_2$ and applying Gram-Schmidt to them. With probability one, we will thus have an orthonormal matrix whose first two rows are the same as the rows of $B_1$. 

The motivation for this construction is as follows. We will later choose $K$ to be very large, so that only what happens in the first two coordinates ``matters'' and the supermodularity of the system reduces to the supermodularity of the system in the first two components (which we already know does not hold by Theorem \ref{first}). 

Let us adopt the notation that for a matrix $M$, we will use $\widehat{M}$ to denote its upper left $2 \times 2$ submatrix. Observe that, by construction, we have for any $K$ that
\begin{equation} \label{kratio}  \frac{K \widehat{W}(A_1(K),B_3(S),+\infty)}{ W(A,B(S),+\infty)} = ~~{\rm constant},\end{equation} where the matrices $A$ and $B$ are  taken from Theorem \ref{first}. Note that division of matrices is here understood elementwise. The constant on the right hand side arises from the fact that the first two rows of $B$ were normalized to obtain $B_1$.

We  now argue that, with probability one, when $K$ is large enough we obtain the counterexample we seek in the pair $A_1(K)$ and $B_3$. The key step is the  identity 
\begin{small}
\begin{equation} \label{matrixid} {\rm tr} \left( \begin{array}{cc} U & V \\ X & Y \end{array} \right)^{-1} = {\rm tr}(U^{-1}) + {\rm tr} \left( (Y - X U^{-1} V)^{-1} (I+X U^{-2} V) \right), \end{equation} \end{small} which holds as long as $U$ is invertible and $Y-X U^{-1} V$ is invertible  \cite{taopost}. Now for any set $S$, let us partition the matrix $W(A_1,B_4(S), +\infty)$ as $\left( \begin{array}{cc} U_W & V_W \\ X_W & Y_W \end{array} \right)$ where its top $2 \times 2$ block is $U_W$.

First observe that, by Eq. (\ref{kratio}), for any $K>0$ the matrix $U_W$ is invertible as long as $S$ is any of the sets in the counterexample of Theorem \ref{first} (i.e., $S=\{1,2\}, S=\{1,2,3,4\}, S=\{1,2,5\}, S=\{1,2,3,4,5\}$), since the corresponding $2 \times 2$ matrices were computed to be invertible in the course of the proof of that theorem. 

Moreover, as $K \rightarrow +\infty$, every nonzero entry of $U_W, V_W, X_W$ goes to zero proportionately to $1/K$, while every entry of $Y_W$ is constant. Thus the matrix $Y_W - X_W U_W^{-1} V_W$ approaches $Y_W$. Since $Y_W$ is invertible with probability $1$ (this can be argued by first observing that it suffices to prove this when $S$ is a singleton; and in that case, it follows from the observation that $Y_W$ is a square submatrix of Hilbert matrix\footnote{The Hilbert matrix is the matrix $H$ defined by $H_{ij} = 1/(i+j-1)$. It is known to be invertible, and indeed an explicit expression for its inverse is available; see for example \url{http://mathworld.wolfram.com/HilbertMatrix.html}.} scaled from the left and right by a random diagonal matrix whose entries have a zero probabiity of equalling zero), we obtain that with probability one,  $Y_W - X_W U_W^{-1} V_W$ is invertible when $K$ is large enough.

Consequently, on the right-hand side of Eq. (\ref{matrixid}) the second term is asymptotically negligible compared to the first one and we obtain 
\[ \lim_{K \rightarrow \infty} \frac{{\rm tr} \left( \begin{array}{cc} U_W & V_W \\ X_W & Y_W \end{array} \right)^{-1}}{{\rm tr}(U_W^{-1})} = 1 \]

Thus, as we choose $K$ large enough, the average control energy of the system $\dot{x} = A_1(K) x + B_3(S) u$ will approach, in ratio, ${\rm tr}~ U_W^{-1}$ which is the same as ${\rm tr} \left[ \widehat{W}(A_1, B_3(S), +\infty)^{-1} \right]$. Now applying Eq. (\ref{kratio}), we see that the ratio of the average control energy of $\dot{x} = A_1(K) x + B_3(S) u$  to 
$K {\rm tr}(W(A,B(S), +\infty)^{-1}$ approaches a constant as $K \rightarrow +\infty$ for any of the sets $S$ used in the proof of Theorem \ref{first}.

In other words, letting $c$ denote the constant of the previous paragraph, we have that as $K \rightarrow +\infty$, the average control energy of $\dot{x}=A_1(K)+B_3(S)$ is $cK{\rm tr}(W(A,B(S), +\infty)^{-1} (1 + o_K(1))$, where $o_K(1)$ denotes something that approaches zero as $K \rightarrow +\infty$. Recall that here $A,B$ are the matrices from Theorem \ref{first}.

 We have already shown, however, the lack of supermodularity for 
${\rm tr}(W(A,B(S), +\infty)^{-1}$ for these sets in Theorem \ref{first}, and thus we conclude that choosing $K$ large enough we can obtain a counterexample to the average control energy $W(A_1,B_3(S),+\infty)$ using the same sets.

\end{proof}

\noindent {\bf Remark:} The matrix $B$ constructed in this example is not uniquely defined, since it relies on the generation of random numbers. However, one run in MATLAB using the ``randn'' command to generate random Gaussians, with the choice of $K=10^4$ yields (after rounding),
\[ A = \left( \begin{array}{cccccc}
-182 & 0 & -565 & 0 & -11 & -736 \\
0 & -1075 & 831 & -276 & -1752 & -612 \\
-565 & 831 & -2435 & 214 & 1321 & -1853 \\
0 & -276 & 214 & -73 & -453 & -158 \\
-11 & -1752 & 1321 & -453 & -2864 & -1045 \\ 
-736 & -612 & -1853 & -158 & -1045 & -3381 
\end{array} \right) 
\] with, of course, $B$ being the $6 \times 6$ identity matrix. Using the MATLAB ``gram'' command to compute controllability Gramians, we obtain
\begin{eqnarray*} \Ea(\{1,2\}) - \Ea(\{1,2,3,4\}) & \approx & 2.50 \cdot 10^5 \\
 \Ea( \{1,2,5\}) - \Ea(\{1,2,3,4,5\}) & \approx & 2.52 \cdot 10^5
 \end{eqnarray*} providing a numerical confirmation of non-supermodularity for this example. 
 
  \smallskip
 
 \noindent {\bf Remark:} It is possible to slightly modify our construction to obtain a $5 \times 5$ counterexample (indeed, perusing through the details of Theorem \ref{first}, it is easy to see that one of the columns of the matrix $B$ is unnecessary). We omit the details.

\section{Conclusion\label{sec:concl}}

We have constructed two examples showing that average control energy is not necessarily a supermodular function of the set of actuated sites or actuated variables. These results are relevant for the problem of actuator placement with average energy constraints, in that they show that a key property that has been used to develop approximation algorithms in other contexts is not available here.  

Indeed in \cite{mincont} it was shown that if actuating the variables in the set $S^*$ renders a system controllable, then one can find in polynomial time a set of size $O(|S^*| \log n)$ that also renders the system controllable, and moreover this is the best possible guarantee one can obtain in polynomial time unless $P=NP$. The proof was based on the submodularity of the dimension of the controllable subspace. It is at present unclear what the analogous best possible guarantee one can attain (in polynomial time) when the control metric is not controllability of the system but rather average control energy.



%\vfill


\end{document}